\documentclass[11pt,a4paper,reqno]{amsart}
\usepackage{amsfonts,amssymb,amsmath,amsthm}
\usepackage{a4wide}
\usepackage{mathrsfs}
\usepackage{xspace}
\usepackage{verbatim}
\usepackage{booktabs}

\newtheorem{theorem}{Theorem}[section]
\newtheorem{proposition}[theorem]{Proposition}

\newtheorem{lemma}[theorem]{Lemma}
\theoremstyle{remark}
\newtheorem{remark}{Remark}[section]
\newtheorem{example}{Example}[section]
\newtheorem{definition}{Definition}[section]

\numberwithin{equation}{section}

\newcommand{\N}{{\mathbb N}}
\newcommand{\R}{{\mathbb R}}
\renewcommand{\S}{{\mathcal S}}
\renewcommand{\L}{\mathscr{L}}

\newcommand{\tx}[1]{\mbox{\;\;{#1}\;\;}}
\newcommand{\beqn}{\begin{eqnarray}}
\newcommand{\eeqn}{\end{eqnarray}}   
\newcommand{\beq}{\begin{eqnarray*}}
\newcommand{\eeq}{\end{eqnarray*}}

\newcommand{\dx}{\mathrm{d}}
\newcommand{\M}{\mathcal{M}}
\title[Boundary singularities of equations with Hardy potential]{Boundary singularities of solutions of\\ semilinear elliptic equations
in the half-space\\
with a Hardy potential}

\author{Catherine Bandle}
\address{Departement Mathematik und Informatik, Universit\"at Basel\\ Spiegelgasse1, CH-4051 Basel, Switzerland}
\email{c.bandle@gmx.ch}

\author{Moshe Marcus}
\address{Department of Mathematics, Technion, Haifa 32000, Israel}
\email{marcusm@math.technion.ac.il}

\author{Vitaly Moroz}
\address{Swansea University\\ Department of Mathematics\\ Singleton Park\\
Swansea\\ SA2~8PP\\ Wales, United Kingdom}	
\email{v.moroz@swansea.ac.uk}

\date{\today}

\begin{document}

\begin{abstract}
We study a nonlinear equation in the half-space $\{x_1>0\}$ with a Hardy potential, specifically
\[-\Delta u -\frac{\mu}{x_1^2}u+u^p=0\quad\text{in}\quad \mathbb R^n_+,\]
where $p>1$ and $-\infty<\mu<1/4$. The admissible boundary behavior of the positive solutions is either
$O(x_1^{-2/(p-1)})$ as $x_1\to 0$, or is determined by the solutions of the linear problem $-\Delta h -\frac{\mu}{x_1^2}h=0$.
In the first part we study in full detail the separable solutions of the linear equations for the whole
range of $\mu$. In the second part, by means of sub and supersolutions we construct separable solutions of the nonlinear problem
which behave like $O(x_1^{-2/(p-1)})$ near the origin and which,
away from the origin have exactly the same asymptotic behavior as the separable solutions of the linear problem.
In the last part we construct solutions that behave like  $O(x_1^{-2/(p-1)})$ at some prescribed
parts of the boundary, while at the rest of the boundary the solutions decay or blowup at a slower rate determined by
the linear part of the equation.
\end{abstract}

\maketitle

\section{Introduction}

We study positive solutions of the semilinear problem
\begin{equation*}
\tag{${P}_\mu$}
\label{P}
-\Delta u -\frac{\mu}{x_1^2} u+u^p=0 \quad \tx{in} \R^n_+,
\end{equation*}
where $p>1$, $-\infty<\mu<1/4$ and $\R^n_+:=\{x\in\R^n:x_1>0\}$ is the half-space in $\R^n$, with $n\ge 2$.

Equation \eqref{P} serves as a model for a more general problem
in a bounded smooth domain $\Omega\subset\R^n$,
\begin{equation*}\label{**}
\tag{${P}_{\Omega,\mu}$}
-\Delta u-\frac{\mu}{\delta^2}u+u^p=0\quad\mbox{in }\Omega,
\end{equation*}
where $\delta(x):=\mathrm{dist}(x,\partial\Omega)$
is the distance function to the boundary of $\Omega$.
Because of the strong singularity of the Hardy potential,
the boundary values of the solutions of \eqref{**} cannot be prescribed arbitrarily.
The Hardy potential forces the solutions either to vanish or to be singular at the boundary.
Existence and boundary behavior of the solutions of \eqref{P} has been discussed in \cite{BaMoRe08,BaMoRe09}.
In \cite{BaMoRe08} it was observed that in some regimes the nonlinearity gives rise to uniform boundary blowup solutions that behave like $O(\delta^{-2/(p-1)})$
at the boundary, whereas for certain values of parameters $\mu$ and $p$  the equation admits solutions that behave like $O(\delta^{\alpha_+})$ or $O(\delta^{\alpha_-})$, uniformly near the boundary, where
\begin{equation}\label{roots}
\alpha_\pm=\frac{1}{2}\pm\sqrt{\frac{1}{4}-\mu}
\end{equation}
are the roots of the quadratic equation $-\alpha(\alpha-1)=\mu$.

The question arises if it is possible to find solutions of \eqref{**} with {\em nonuniform} boundary behavior,
say solutions that grow like $O(\delta^{-2/(p-1)})$ on one part of the boundary and like $O(\delta^{\alpha_\pm})$ on other parts.
We refer to solutions that grow as $O(\delta^{-2/(p-1)})$ at parts of the boundary as {\em solutions with strong singularities}.
Note that {\em moderate} singular solutions of \eqref{**} that behave like $O(\delta^{\alpha_-})$ on some parts
of the boundary and $O(\delta^{\alpha_+})$ at another part were recently studied in \cite{Marcus-Nguen,MM}.
The existence of solutions of \eqref{**} with a nonuniform boundary behavior was also  discussed in \cite{GkVe15}.

The aim of the present paper is to study positive solutions with strong singularities in the case of the model problem \eqref{P} on the half-space.
In what follows, we say that a solution $u$ of \eqref{P} has a {\em strong singularity} at a point $x_0\in\partial\R^n_+$ if
$$\liminf_{x\to x_0}u(x)|x-x_0|^\frac{2}{p-1}>0\quad\text{non-tangentially}.$$
If  $\frac{2(p+1)}{(p-1)^2} + \mu>0$, it is straightforward to see that the function
\begin{equation}\label{e-Ustar}
U_*(x)=C_{p,\mu}x_1^{-\frac{2}{p-1}}
\end{equation}
with $C_{p,\mu}= \left\{\frac{2(p+1)}{(p-1)^2} + \mu\right\}^{1/(p-1)}$
is a solution \eqref{P} that has a strong singularity at every point of the boundary.
The above condition on $\mu$ is fulfilled when $1<p<p_{KO}$, where
\begin{equation}\label{de:p}
p_{KO}:=
\begin{cases}
+\infty & \mbox{if } \mu\ge 0,\\
1-\frac{2}{\alpha_-} &\mbox{if } \mu < 0,\\
\end{cases}
\end{equation}
is the {\em Keller--Osserman exponent}.
Note that if $\mu<0$ then $\alpha_-<0$, and hence $p_{KO}>1$.
An adaptation of the arguments in \cite[Lemma 4.11]{BaMoRe08} shows that for $\mu<0$ the critical exponent $p_{KO}$ is sharp,
in the sense that for $p\ge p_{KO}$ problem \eqref{P} has no solutions that behave like $O(x_1^{-2/(p-1)})$ near the boundary.

If $\mu>\frac{1}{4}$, Lei Wei \cite{Wei15} has proved that $U_*$ is the unique positive solution of problem \eqref{P}.
In this case the blowup rate of the solution near the boundary is determined only by the nonlinearity and does not depend on $\mu$.
See also \cite{DuWei15} for relevant results in the case of problem \eqref{**}.

In what follows we focus on the case $-\infty<\mu<\frac{1}{4}$. Then the linear regime comes into play,
when the behavior of positive solutions of \eqref{P} is determined not only by the nonlinearity but also by the linear equation
\begin{equation*}\tag{${L}_\mu$}
\label{harmonic}
-\L_\mu h=0 \quad \tx{in} \R^n_+, \quad\tx{where} \L_\mu:=\Delta+\frac{\mu}{x_1^2}.
\end{equation*}
The restriction $\mu<1/4$ is related to Hardy's inequality
$$\int_{\R^n_+}|\nabla \varphi|^2\ge \frac{1}{4}\int_{\R^n_+}\frac{\varphi^2}{x_1^2},\qquad\forall\varphi\in C^\infty_c(\R^n_+),$$
where $1/4$ is the {\em optimal constant}, see \cite{MMP}. It turns out that for $\mu>\frac{1}{4}$ no positive
solutions of \eqref{harmonic} exists \cite{MMP},
while for $\mu<1/4$ the set of positive solutions \cite{MMP}
has a reach structure.

In what follows, a function $h\in L^1_{loc}(\R^n_+)$ is called an $\L_\mu$-{\em subharmonic function}
if $h$ is a distributional subsolution of \eqref{harmonic}, i.e.,
\begin{equation*}
\int_{\R^n_+}h(-\Delta \varphi)\,dx-
\int_{\R^n_+}\frac{\mu}{x_1^2}h\varphi\,dx\le\:0
\qquad\forall\:0\le\varphi\in C^\infty_c(\R^n_+).
\end{equation*}
If we replace ``$\le$'' by ``$\ge$'' or ``$=$'' in the above inequality
then we say that $h$ is an $\L_\mu$-{\em superharmonic} or $\L_\mu$-{\em harmonic}, respectively.
By the elliptic regularity, every $\L_\mu$-{\em harmonic} is a classical solution of \eqref{harmonic}.
A {\em separable} $\L_\mu$-harmonic is of the form $r^\gamma H(\theta)$ where $r=|x|$ is the distance to the origin and $\theta$ is the azimuth,
that is $\cos(\theta)=\frac{x_1}{|x|}$, see Section \ref{sect-polar} below for detailed definitions.

If $\mu<1/4$, then
$$h_+(x_1)=x_1^{\alpha_+}\quad\text{and}\quad h_-(x_1)=x_1^{\alpha_-}$$
are two $\L_\mu$-harmonics with uniform decay at the boundary.
We call $h_+$ {\em small} and $h_-$ {\em large} $\L_\mu$-harmonics, respectively.
Observe that $\alpha_->0$ if $0<\mu<1/4$ and large $\L_\mu$-harmonic $h_-$ vanishes at the boundary.
Nevertheless $\eta h_-\not\in H^1_0(\R^N_+)$, for any $\eta\in C^\infty_c(\R^n)$ that is positive on a part of $\{x_1=0\}$.

A direct computation also shows that the function
\begin{equation*}
H_-(x)=x_1^{\alpha_+}|x|^{-(n-2+2\alpha_+)}
\end{equation*}
is a separable positive $\L_\mu$-harmonic. $H_-$ has an isolated point singularity at the origin,
and behaves like the small $\L_\mu$-harmonic $h_+$ away from the origin.
It is known that $H_-$ is the unique (up to a scalar multiple) positive $\L_\mu$-harmonic with this behavior \cite{Ancona,Ancona-88}.
In Proposition \ref{thm:h1} we provide a uniqueness proof in the class of separable $\L_\mu$-harmonics,
which is based only on elementary ODE arguments.

The set of positive $\L_\mu$-harmonics with a point singularity at the origin that behave like
the large $\L_\mu$-harmonic $h_-$  away from the origin has a more complicated structure.
We show that the admissible rate of growth at the origin for such $\L_\mu$-harmonics fills an entire interval.

\begin{theorem}\label{t-phs}
Let $-\infty<\mu<1/4$.
For every $\gamma\in\big(-(n-2+\alpha_+),\alpha_+\big)$
there exists a separable positive $\L_\mu$-harmonic $H_\gamma$ such that
\begin{equation}\label{e-gamma-h}
c^{-1} x_1^{\alpha_-}|x|^{\gamma-\alpha_-}\leq H_\gamma(x)\leq c x_1^{\alpha_-}|x|^{\gamma-\alpha_-}\qquad \forall x\in \R^n_+.
\end{equation}
\end{theorem}

\begin{remark}
For $\gamma=\alpha_-$ and $\gamma=-(\alpha_-+n-2)$ the $\L_\mu$-harmonic $H_\gamma$ has the explicit form
\begin{equation*}
H_{\alpha_-}=x_1^{\alpha_-} \tx{and} H_{-(\alpha_-+n-2)}= x_1^{\alpha_-}|x|^{-(n-2+2\alpha_-)}.
\end{equation*}
In the Examples \ref{ex-1} and \ref{ex-2} we  point out its connection with the pure Laplacian, i.e. $\mu=0$.
\end{remark}

\smallskip

Next we move to the study of the nonlinear problem \eqref{P}.
By a solution of \eqref{P} in what follows we understand a function $u\in L^1_{loc}(\R^n_+)$
that satisfies \eqref{P} in the distributional sense. Similarly, subsolutions and supersolutions are defined,
if we replace ``$=$'' by ``$\le$'' and ``$\ge$'', respectively.
By the elliptic regularity, every distributional solution of \eqref{P} is also a classical solution of \eqref{P}.
A {\em separable} solution is of the form $r^\gamma v(\theta)$, where $r$ is the distance to the origin and $\theta$ is the azimuth.

Our first result is the existence of a {\em separable} solution to \eqref{P} with a strong singularity
at the origin that behaves like a small $\L_\mu$-harmonic at the boundary, away from the origin.
We show that such solutions exist below the {\em critical exponent}
\begin{equation}\label{pc}
p_c:=1+\frac{2}{n-2+\alpha_+}.
\end{equation}
Observe that for any $\mu<1/4$ we always have $1<p_c<p_{KO}\le\infty$.

\begin{theorem}\label{thm-main}
Let $-\infty<\mu< 1/4$ and $1<p<p_c$. Then \eqref{P} admits a unique separable solution $u$ such that
\begin{equation}\label{e-plus}
c^{-1} x_1^{\alpha_+}|x|^{-\alpha_+ - \frac{2}{p-1}}\leq u(x)\leq c x_1^{\alpha_+}|x|^{-\alpha_+ - \frac{2}{p-1}} \qquad \forall x\in \R^n_+.
\end{equation}
\end{theorem}

Relevant results for $0<\mu<1/4$ were also established in \cite[Section 4]{GkVe15}.
\medskip

Using an elementary Phragm\'en--Lindel\"of type argument in combination with a localised Keller--Osserman type bound,
we also establish a nonexistence result that shows that the value of the critical exponent $p_c$ is sharp.
\begin{theorem}\label{thm-nonexist}
Let $-\infty<\mu< 1/4$ and $p>p_c$. Then \eqref{P} does not admit positive solutions that satisfy \eqref{e-plus}.
\end{theorem}

\begin{remark}
Using potential theoretical techniques in the spirit of \cite{Marcus-Nguen,MM} one can extend the above nonexistence result
to the critical case $p=p_c$. However this would go beyond the scope of the present work.
\end{remark}

Next we construct separable solutions with a strong singularity at the origin, that behave like large $\L_\mu$-harmonics at the rest of the boundary.
Such solutions exist for a wider range $1<p<p_{KO}$ and they are not unique.
\begin{theorem}\label{thm-main-minus}
Let $-\infty<\mu<1/4$ and $1<p<p_{KO}$. Then \eqref{P} admits a positive solution $u$ such that
\begin{equation}\label{e-minus}
c^{-1} x_1^{\alpha_-}|x|^{-\alpha_- - \frac{2}{p-1}}\leq u(x)\leq c x_1^{\alpha_-}|x|^{-\alpha_- -\frac{2}{p-1}} \qquad \forall x\in \R^n_+.
\end{equation}
\end{theorem}

In view of the Keller--Osserman a priori bound, which is also valid for \eqref{P} (see Proposition \ref{p-KO}), this result is sharp,
as for $p>p_{KO}$ equation \eqref{P} does not admit positive solutions with strong singularities.
\begin{theorem}\label{t-non-KO}
Let $-\infty<\mu<1/4$ and $p>p_{KO}$.
Then \eqref{P} does not admit positive solutions which satisfy \eqref{e-minus}.
\end{theorem}

\begin{remark}
For every $p\in(p_c,p_{KO})$ the solution $u=r^{-\frac{2}{p-1}}v(\theta)$ constructed in Theorem \ref{thm-main-minus} is dominated by the $\L_\mu$-harmonics $H_\gamma$ with $\gamma={-\frac{2}{p-1}}$, constructed in Theorem \ref{t-phs}. This follows directly from the comparison of the bounds \eqref{e-minus} and \eqref{e-gamma-h}.
Therefore, in the {\em supercritical} range $p_c< p< p_{KO}$ the solution $u$ is a {\em moderate} solution of the nonlinear problem \eqref{P},
in the sense of \cite{Marcus-Nguen} or \cite{MM}.
At the same time, the solution $U_*=C_{p,\mu}x_1^{-\frac{2}{p-1}}$ described in \eqref{e-Ustar} is not dominated by a positive $\L_\mu$-harmonic for any $p\in(1,p_{KO})$. To see this, recall that as a consequence of the Phragm\'en--Lindel\"of comparison principle, any positive $\L_\mu$-harmonic $h$ must satisfy $\liminf_{x_1\to 0}h(x)x_1^{-\alpha_-}<+\infty$, cf. \cite[Theorem 2.6]{BaMoRe08}, which contradicts to $p<p_{KO}$.
\end{remark}

Finally, we show that for the subcritical values of $1<p<p_c$ one can construct
solutions with strong singularities on arbitrary compact subsets of the boundary.

\begin{theorem}\label{thm-general}
Let $-\infty<\mu<1/4$ and $1<p<p_c$.
Let $F\subset\partial\R^n_+$ be a closed set.
Then there exists a solution $U$ of \eqref{P} such that $U=0$ on $\partial\R^n_+\setminus F$ and, for every $(0,\xi)\in F$,
\begin{equation}\label{gen1}
c^{-1}x_1^{-\frac{2}{p-1}}\leq U(x_1,\xi)\leq cx_1^{-\frac{2}{p-1}} \quad\forall x_1\in (0,1),
\end{equation}
where $c>0$ depends only on $\mu,p$. Furthermore,
\begin{equation}\label{gen2}
c^{-1}x_1^{\alpha_+}\mathrm{dist}(x,F)^{-\alpha_+ -\frac{2}{p-1}}\leq U(x)\leq c x_1^{\alpha_+}\mathrm{dist}(x,F)^{-\alpha_+ -\frac{2}{p-1}}, \quad  \quad \forall x\in \R^n_+
\end{equation}
with $c$ as before.
\end{theorem}

The separable solutions play an important role in our consideration.
In Section \ref{s-sep-lin} we start with the construction and discussion of separable $\L_\mu$-harmonics.
Our main result here is a detailed characterisation of admissible singularities at the origin
of separable $\L_\mu$-harmonics which are small (Proposition \ref{thm:h1}) and large (Proposition \ref{phs})
on the boundary and away from the origin. A particular case was presented in Theorem \ref{t-phs}
but we believe that the consideration of $\L_\mu$-harmonics in Section \ref{s-sep-lin} could be of independent interest.
\smallskip

In Section \ref{s-sep-nonlin} we construct separable solutions for the nonlinear problem \eqref{P}
that satisfy the assumptions of Theorems \ref{thm-main} and \ref{thm-main-minus}.
In Section \ref{s-PL} we prove nonexistence results of Theorems \ref{thm-nonexist} and \ref{t-non-KO},
while in Section \ref{s-general} we present some results on solutions with a general singular set
and prove Theorem \ref{thm-general}.

\section{Construction of separable $\L_\mu$-harmonics}\label{s-sep-lin}
\subsection{Polar coordinates}\label{sect-polar}
We introduce the polar coordinates
$$
\left\{
\begin{array}{l}
x_n=r(\prod_{j=1}^{n-2}\sin\theta_j)\sin\phi,\vspace{10pt}\\
x_{n-1}=r(\prod_{j=1}^{n-2}\sin\theta_j)\cos\phi, \vspace{10pt}\\
x_k=r(\prod_{j=1}^{k-2}\sin\theta_j)\cos\theta_k,\quad k=3,4,\dots, n-2,\vspace{10pt}\\
x_1=r\cos\theta_1,
\end{array}
\right.
$$
where $r>0$, $0\leq \theta_j\le \pi$ ($j=1,2,3,\dots, n-2$) and
$0\le \phi\leq 2\pi$. In these coordinates the Laplacian is expressed as
$$
\Delta =\Delta_r +\frac{1}{r^2} \Delta_{S^{n-1}},
$$
where
$ \Delta_r =r^{1-n}\partial_r( r^{n-1}\partial _r)$ and $\Delta_{\S^{n-1}}$ is the Laplace-Beltrami operator on
${\mathcal{S}}^{n-1}$.  It can be determined recursively as follows. Set  $t=\cos(\theta_1)$ and let $\eta\in \S^{n-2}$. Then for $n\geq 2$, $x=(rt,r\sqrt{1-t^2}\eta)$ and
$$
\Delta_{\S^{n-1}}=\underbrace{\frac{1}{(1-t^2)^{\frac{n-3}{2}}} \frac{\partial}{\partial t}\left((1-t^2)^{\frac{n-1}{2}}\frac{\partial}{\partial t}\right)}_L+\frac{1}{1-t^2}\Delta_{\S^{n-2}}.
$$
Our goal is to look for solutions of \eqref{harmonic} and \eqref{P} of the form
$$r^\gamma k(t)p(\eta).$$
They will be called {\em separable} $\L_\mu$-harmonics and {\em separable} solutions, respectively.
\subsection{Separable $\L_\mu$-harmonics}
In what follows we denote
\begin{equation}\label{e-Lambda}
\Lambda(\gamma):=\gamma(\gamma+n-2).
\end{equation}
Since $x_1=rt$, separable $\L_\mu$-harmonics
$$h(x)=r^\gamma k(t)p(\eta)$$
of \eqref{harmonic} satisfy,
$$
\frac{L k}{k} +\frac{\mu}{t^2}+ \frac{1}{1-t^2}\frac{ \Delta_{S^{n-2}}p}{p}+\Lambda(\gamma) =0.
$$
This implies that
$$
\frac{ \Delta_{\S^{n-2}}p}{p}=\rm{const}\quad  \text{on $\S^{n-2}$}.
$$
If $n=2$, this term does not exist.
Since we are looking for solutions in the half-space,  $p$ is an eigenfunction on $\S^{n-2}$. It is well-known that the eigenvalues are
$$\nu_m:=m(m+n-3),\quad m=0,1,2,\dots.$$
Replacing $p^{-1} \Delta_{\S^{n-2}}p$ by $-\nu_m$, we obtain  for $k$ the differential equation
\begin{align}\label{eq:k}
(1-t^2)\ddot{k}&-(n-1)t\dot{k} + \frac{\mu}{t^2} k +  k\left(\Lambda(\gamma) -\frac{\nu_m}{1-t^2}\right)=0 \tx{in} (0,1), \\
\nonumber &\tx{where} \dot{}=\frac{d}{dt}.
\end{align}
This equation is of Fuchsian type and has singularities at $t=0$ and at $t=\pm1$, see \cite[Chapter 4]{Hartman}.
Its solutions can be expressed in terms of hypergeometric functions.
Since in our case $t$ varies in $[0,1]$, we are only interested in the local behavior of the solutions in $t=0$ and $t=1$.
It is determined by the indicial equations.

The indicial equation at $t=0$ is
\begin{align}\label{indicial0}
\alpha(\alpha-1)+\mu=0.
\end{align}
As in \eqref{roots}, its roots are given by
$$
\alpha_\pm:=\frac{1}{2} \pm \sqrt{\frac{1}{4}-\mu}.
$$
Because for $\mu<1/4$, $\alpha_-<\alpha_+$, there exists a regular solution $k(t)$ near $t=0$ of the form $t^{\alpha_+} f_0(t)$,
where $f_0(t)$ is an analytic function in $(0,1)$ such that $f_0(0)\neq 0$.
In addition there is a singular solution near $t=0$ of the form $t^{\alpha_-} f_1(t)$ where $f_1(t)$ is an analytic function in $[0,1)$ such that $f_1(0)\neq 0$.
\medskip

The indicial equation at $t=1$ is
\begin{align*}
\kappa^2 +\kappa \left(\frac{n-3}{2}\right) -\frac{\nu_m}{4}=0.
\end{align*}
Its roots are
$$
\kappa_\pm= \pm\sqrt{\frac{\nu_m}{4} +\left(\frac{n-3}{4}\right)^2} -\frac{n-3}{2}.
$$
Observe that for $m\neq 0$, $\kappa_+>0$ which implies that $k(t)$ vanishes at $t=1$.
We are looking for a separable $\L_\mu$-harmonic which is regular in $\mathbb{R}^n_+$. This means that we are only interested in those solutions which behave like $k(t)\sim(1-t)^{\kappa_+}$ near $t=1$.
\medskip

At the origin $k(t)$ behaves either like $t^{\alpha_+}$ or  $t^{\alpha_-}$. We say that $k(t)$ is {\em regular, singular} at the origin, respectively. We shall discuss the two cases separately.
\subsubsection{Separable $\L_\mu$-harmonics that are regular on $\partial \mathbb{R}^n_+\setminus \{0\}$}
Problem \eqref{eq:k} with $\gamma\in\R$ can be written as
\begin{equation}\label{eq:k'}
\frac{d}{dt} \left (\sigma(t)\dot k_\gamma\right) +\frac{\sigma(t)}{1-t^2}\left( \frac{\mu}{t^2} +\Lambda (\gamma)-\frac{\nu_m}{1-t^2}\right)k_\gamma =0,\tx{where}\sigma(t):= (1-t^2)^{\frac{n-1}{2}}.
\end{equation}
This equation could be interpreted as an eigenvalue problem, weakly formulated as
\begin{align}\label{eq:eigenvalue}
\int_0^1\sigma \dot{\varphi}\dot \psi\:dt-\int_0^1\frac{\sigma(t)}{1-t^2} \left(\frac{\mu}{t^2}-\frac{\nu_m}{1-t^2}\right)\varphi\psi\:dt=\Lambda(\gamma)\int_0^1\frac{\sigma(t)}{1-t^2} \varphi\psi\:dt,
\end{align}
for all test functions $\psi$, where $\varphi$ and $\psi$ are in the weighted Sobolev space $W^{1,2}(0,1;\sigma, \frac{\sigma}{1-t^2})$  and are such that all integrals exist, that is they vanish at both endpoints if $\nu_m\neq 0$, and only at $t=0$ if $\nu_m=0$.

From the classical spectral theory and since $\mu<1/4$,  it follows that for each $\nu_m$, $m=0,1,2,\dots$,  there exists a countable sequence of eigenvalues
of \eqref{eq:eigenvalue},
$$0<\Lambda_{1,m}<\Lambda_{2,m}\leq \cdots.$$
It is easy to see that $\varphi(t)=t^{\alpha_+}$ is a solution of \eqref{eq:eigenvalue} for $\Lambda=\Lambda(\alpha_+)$ and  $\nu_m=0$. Since it does not change sign it is the eigenfunction corresponding to $\Lambda_{1,0}$ and it is the lowest eigenvalue of \eqref{eq:eigenvalue}. If $n=2$, then $\nu_m=0$ and we have only one series of eigenvalues $\Lambda_{j,0}$, $j=1,2\dots.$

\begin{definition}
$\mathcal{H}_{\mu,0}$ denotes the class of separable $\L_\mu$-harmonics $h$ that for any $R>0$ behave like
$$c^{-1}x_1^{\alpha_+}\leq h(x)\leq c x_1^{\alpha_+} \tx{as $R^{-1}\le|x|\le R$ and $x_1\to 0$.}$$
Here $c$ is a positive constant depending in general on $R$.
\end{definition}

We are now in a position to construct $\L_\mu$-harmonics in $\mathcal{H}_{\mu,0}$.
\begin{enumerate}
\item Choose $m\in \N_0:=\mathbb{N}\cup\{0\}$ and let  $p_m(\eta)$ be an eigenfunction of $\S^{n-2}$ corresponding to the eigenvalue $\nu_m= m(m+n-3)$. (The multiplicity of $\nu_m$ is, except for $m=0$, larger than one).
\item Choose $s\in \mathbb{N}\cup\{0\}$ and let $\varphi_{s,m}$ be the eigenfunction corresponding to $\Lambda_{s,m}$.
\item Determine $\gamma$ such that $\Lambda(\gamma)=\Lambda_{s,m}$. This leads to two roots
\begin{equation}\label{gamma}
\gamma_\pm(s,m):= \pm\sqrt{\Lambda_{s,m} +\left(\frac{n-2}{2}\right)^2}-\frac{n-2}{2}.
\end{equation}
\item Then the functions
\begin{align}\label{H_0}
h_+= r^{\gamma_+}\varphi_{s,m}(\cos(\theta_1))p_m(\eta) \tx{and} h_-=r^{\gamma_-}\varphi_{s,m}(\cos(\theta_1))p_m(\eta)
\end{align}
belong to $\mathcal{H}_{\mu,0}$.
\end{enumerate}
Since the indicial equation \eqref{indicial0} is independent of $\Lambda(\gamma)$, all eigenfunctions $\varphi_{s,m}$ behave at the boundary like $t^{\alpha_+}$.
In summary we have obtained the following.
\begin{proposition}\label{thm:h1}
(i) All separable $\L_\mu$-harmonics in $\mathcal{H}_{\mu,0}$ are given by \eqref{H_0}.
Near the boundary and away from the origin, we have
$$\lim_{\theta_1 \to \pi/2}\frac{\varphi_{s,m}(\cos(\theta_1))}{\cos^{\alpha_+}(\theta_1)} = c>0,$$
for all $(s,m)\in\N_0\times\N_0$.
If $m\neq 0$ then we have $\varphi_{s,m}(1)=0$ and therefore $h_\pm$ vanishes on the whole $x_1-$axis.
In particular, all $\L_\mu$-harmonics corresponding to $(s,m)\neq (1,0)$ change sign.
\medskip

(ii) The only positive separable $\L_\mu$-harmonics in $\mathcal{H}_{\mu,0}$ are given by
\begin{equation}\label{harmonics1}
h_+=x_1^{\alpha_+} \tx{and}
h_-=r^{-(\alpha_++n-2)}\cos(\theta_1)^{\alpha_+}.
\end{equation}
\end{proposition}
\begin{remark} Since $\alpha_+>0$ the function $h_-$ in \eqref{harmonics1} always has a point singularity at the origin.
\end{remark}
\subsubsection{Separable $\L_\mu$-harmonics that are singular at $\partial \mathbb{R}^n_+\setminus \{0\}$}
Similar to the class $\mathcal{H}_{\mu,0}$ we introduce
\begin{definition}
$\mathcal{H}_{\mu,sing}$ is the class of separable $\L_\mu$-harmonics that for any $R>0$
behave like
$$c^{-1}x_1^{\alpha_-}\leq h(x)\leq c x_1^{\alpha_-} \tx{as $R^{-1}\le|x|\le R$ and $x_1\to 0$.}$$
Here $c$ is a positive constant depending in general on $R$.
\end{definition}

We construct separable $\L_\mu$-harmonics in $\mathcal{H}_{\mu,sing}$ for values of $\gamma$ such that $\Lambda(\gamma)\neq \Lambda_{s,m}$ (cf. \eqref{gamma}).
We shall make use of the solutions of \eqref{eq:k} that are regular at $t=1$.

If we integrate \eqref{eq:k'} from $t$ to $1$ and keep in mind that the  solutions which are regular on the $x_1$-axis satisfy $\sigma(1)\dot{k}_\gamma(1)=0$,
we obtain
$$
-\sigma(t)\dot{k}_\gamma(t) +\int_t^1\frac{\sigma(s)}{1-s^2}\left( \frac{\mu}{s^2} +\Lambda(\gamma)-\frac{\nu_m}{1-s^2}\right) k_\gamma(s)\mathrm{d}s=0,
$$
and after a second integration
\begin{align}\label{eq:kint}
k_\gamma(t)=k_\gamma(1)-\int_t^1\sigma^{-1}(\xi)\dx \xi\int_\xi^1\frac{\sigma(s)}{1-s^2}\left( \frac{\mu}{s^2} +\Lambda(\gamma)-\frac{\nu_m}{1-s^2}\right) k_\gamma(s)\dx s=:T(k).
\end{align}
It is not difficult to see that for given $k_\gamma(1)>0$ if $m=0$ or $k_\gamma(1)=0$ otherwise, and for $t_0$ near $t=1$, $T(k)$ is a contraction in the Banach space $(C[t_0,1], \|\cdot\|_\infty)$. Consequently \eqref{eq:kint} has a unique solution $k_\gamma(t)$ in $[t_0,1]$. This solution can be continued into the whole interval $(0,t_0)$. Either it vanishes at $t=0$ and behaves like $t^{\alpha_+}f_0(t)$ or it becomes singular like $t^{\alpha_-}f_1(t)$.  In the first case $k_\gamma$ is an eigenfunction of \eqref{eq:eigenvalue}. It can be excluded by our assumption on $\Lambda(\gamma)$. Hence $k_\gamma(t)$ is singular on $\partial \mathbb{R}^n_+\setminus \{0\}$ and the corresponding separable $\L_\mu$-harmonics
\begin{equation}\label{sharmonics}
H_+= r^\gamma k_\gamma(\cos(\theta_1))p_m(\eta) \tx{and} H_-=r^{-(\gamma+n-2))} k_\gamma(\cos(\theta_1))p_m(\eta),
\end{equation}
belong to $\mathcal{H}_{\mu,sing}$.

\begin{proposition}\label{phs}
(i) Let  $\gamma\in \mathbb{R}$ be such that $\Lambda(\gamma)\neq \Lambda_{s,m}$ for any $(s,m)\in\N_0\times\N_0$.
Then the functions $H_\pm$ constructed in \eqref{sharmonics} are $\L_\mu$-harmonics. They are singular near the boundary and away from the origin in the sense that
$$\lim_{\theta_1 \to \pi/2}\frac{k_\gamma(\cos(\theta_1))}{\cos^{\alpha_-}(\theta_1)} = c>0.$$
\medskip

(ii) If $\Lambda(\gamma)<\Lambda(\alpha_+)$, then
\begin{equation}\label{eq:g}
H_+=r^\gamma k_\gamma(\cos(\theta_1)) \tx{and} H_-=r^{-(\gamma+n-2)} k_\gamma(\cos(\theta_1))
\end{equation}
are positive separable $\L_\mu$-harmonics in $\mathcal{H}_{\mu,sing}$.
\medskip

(iii) For $\gamma=\alpha_-$ we have the particular positive separable $\L_\mu$-harmonics in $\mathcal{H}_{\mu,sing}$ of the form
\begin{equation*}
H_+=x_1^{\alpha_-} \tx{and} H_-= r^{-(\alpha_-+n-2)}\cos^{\alpha_-}(\theta_1).
\end{equation*}
\end{proposition}

\begin{remark}
Any linear combination of $\L_\mu$-harmonics is again an $\L_\mu$-harmonic, but it is not necessarily a separable $\L_\mu$-harmonic.
\end{remark}

\begin{remark} It follows also from Sturm's comparison theorem that every $\L_\mu$-harmonic of the form \eqref{sharmonics} is positive if $\Lambda(\gamma)<\Lambda(\alpha_+)$.
The admissible range of $\gamma$ for positive $\L_\mu$-harmonics in \eqref{eq:g} fills therefore the whole interval $\gamma\in(-(\alpha_++n-2),\alpha_+)$. Observe that $\alpha_-=-\alpha_++1$ belongs to this range. This is in accordance with the statement in Proposition \ref{phs} (iii).
\end{remark}

Clearly, Theorem \ref{t-phs} is a special case of Propositions \ref{phs}.
In the case of the pure Laplacian, i.e. $\mu=0$, we will illustrate positive harmonics constructed in Propositions \ref{thm:h1} and \ref{phs} by two examples.

\begin{example}\label{ex-1}
Consider the special case $n=2$ and $\mu=0$. Then $\alpha_+=1$, $\alpha_-=0$, $\Lambda(\alpha_+)=1$ and $\Lambda(\alpha_-)=0$.
The separable $\L_\mu$-harmonics  constructed in this section are $H_+=r^\gamma \cos(\gamma\phi)$ and $H_-=r^{-\gamma} \cos(\gamma \phi)$,
where $-\frac{\pi}{2}<\phi<\frac{\pi}{2}$. They belong to $\mathcal{H}_{0,0}$ if  $\gamma=1+2s$ for $s\in\N_0$ and otherwise to $\mathcal{H}_{0,sing}$.
If $-1\leq \gamma \leq 1$ they are positive in $\mathbb{R}^2_+$. The class $\mathcal{H}_{0,sing}$ consists of the functions $r^\gamma\cos(\gamma \phi)$ with $-1<\gamma<1$.
They can be conceived as the restrictions of positive harmonics in sectors with the opening angle $\phi\in (-\frac{\pi}{2\gamma},\frac{\pi}{2\gamma})$, vanishing on the boundary.
\end{example}

\begin{example}\label{ex-2}
Now consider the case $\mu=0$ and $n>2$. Then $\alpha_+=1$, $\alpha_-=0$, $\Lambda(\alpha_+)=n-1$ and $\Lambda(\alpha_-)=0$.
The admissible $\gamma$-range for separable positive harmonics $H_\gamma$  constructed above with $\mu=0$
is $\gamma\in[-(n-1),1]$. Unlike on the plane, not all the positive harmonics in $\mathcal{H}_{0,sing}$ are restrictions of the positive harmonics
in cones containing the half space. This difference is discussed next.

First consider the Laplacian $-\Delta$ on the conical domain $\mathcal C_\vartheta=\{x\in\R^n\setminus\{0\}\,;\,\theta_1<\vartheta\}$,
where $\vartheta\in[\pi/2,\pi)$.
Let $\lambda_1(\mathcal C_\vartheta)>0$ be the positive principal Dirichlet eigenvalue of $-\Delta_{\mathcal S^{n-1}}$ on the cross--section $\mathcal S^{n-1}\cap \mathcal C_\vartheta$
and $\phi_1\in H^1_0(\mathcal S^{n-1}\cap \mathcal C_\vartheta)$ be a corresponding positive principal Dirichlet eigenfunction.
It is not difficult to see that $\lambda_1(\mathcal C_\vartheta)$ is a continuous monotone decreasing function.
Since $\lambda_1(\mathcal C_{\pi/2})=n-1$ and $\lim_{\vartheta\to\pi}\lambda_1(\mathcal C_\vartheta)=0$,
a direct computation shows that for every $\gamma\in(0,1)$ there exits $\vartheta\in(\pi/2,\pi)$ such that
$$H_+=r^{\gamma}\phi_1,\qquad H_-=r^{-(\gamma+n-2)}\phi_1,$$
are harmonic on the cone $\mathcal C_\vartheta$ (see e.g. \cite{DPP} for similar constructions on conical domains).
The restriction of $H_\pm$ on $\R^n_+\subset\mathcal C_\vartheta$ are harmonics in the class $\mathcal H_{0,sing}$ of the type constructed in \eqref{eq:g} with $\gamma\in(0,1)\cup\big(-(n-1),-(n-2)\big)$.

To cover the remaining range $\gamma\in[-(n-2),0]$, consider in $\R^n\setminus\{0\}$ the operator
\begin{equation*}
-L_\vartheta:=-\Delta - \frac{\vartheta a(\theta_1)}{|x|^2},
\end{equation*}
where $\vartheta\ge 0$ and $a:[0,\pi]\to \R$ is a nonnegative continuous function such that $a(\theta_1)=0$ for $\theta_1\in[0,\pi/2]$ and $a(\pi)=1$.
Let $\lambda_1(\vartheta a)\le 0$ denotes the principle eigenvalue of the Laplace Beltrami operator
$-\Delta_{\mathcal S^{n-1}}-\vartheta a$ on $\mathcal S^{n-1}$ and $\phi_1\in H^1(\mathcal S^{n-1})$ be the corresponding positive eigenfunction.
Again, it is not difficult to see that $\lambda_1(\vartheta a)$ is a continuous monotone decreasing function,
$\lambda_1(0)=0$ and $\lim_{\vartheta\to\infty}\lambda_1(\vartheta a)=-\infty$,
In particular, there exists a critical value $\vartheta^*>0$ such that $\lambda_1(\vartheta^*a)=-\frac{(n-2)^2}{4}$.
A direct computation then shows that for every $\gamma\in[-\frac{n-2}{2},0]$ there exits $\vartheta\in[0,\vartheta^*]$ such that
$$H_+=r^{\gamma}\phi_1,\qquad H_-=r^{-(\gamma+n-2)}\phi_1,$$
are $L_\vartheta$-harmonics in $\R^n\setminus\{0\}$.
Since $a(\theta_1)=0$ for all $\theta_1\in[0,\pi/2]$,
the restriction of $H_\pm$ on $\R^n_+$ are $\Delta$-harmonics in the class $\mathcal H_{0,sing}$ of the type constructed in \eqref{eq:q} with $\gamma\in[-(n-2),0]$.
\end{example}


\section{Separable solutions of the nonlinear problem}\label{s-sep-nonlin}
Our goal is to look for separable solutions of the nonlinear equation \eqref{P} of the form
$$u(r,\theta)= r^{-\frac{2}{p-1}}v(t).$$
Observe that in contrast to the separable $\L_\mu$-harmonics, $u$ is independent of $\eta$. The equation for $v(t)$ is
\begin{align}\label{eq:v}
\mathcal{L}v:=(1-t^2)\ddot{v} -(n-1) t\dot{v} +\frac{\mu}{t^2}v + \Lambda\left(-\tfrac{2}{p-1}\right)v= v^p
\qquad\text{for $t\in(0,1)$,}
\end{align}
where $\Lambda(\cdot)$ is defined in \eqref{e-Lambda}.
The existence of a solution is based on the method of upper and lower solutions. The
following identity will play a crucial role. Let $\alpha=\alpha_\pm$, then
\begin{align}\label{identity}
\mathcal{L}t^{\alpha +\epsilon}
= \epsilon(2\alpha +\epsilon-1)t^{\alpha + \epsilon-2} +\left[\Lambda\left(-\tfrac{2}{p-1}\right)-\Lambda(\alpha+\epsilon)\right]t^{\alpha +\epsilon}.
\end{align}

\subsection {Construction of a solution $v$ that behaves like $t^{\alpha_+}$ near $t=0$}

\noindent  Assume
\begin{align*}
1<p<p_c,
\end{align*}
where $p_c$ is defined by \eqref{pc}.
Then $\Lambda\left(-\tfrac{2}{p-1}\right)> \Lambda(\alpha_+)$. By \eqref{identity} the function
$\underline v=\tau t^{\alpha_+}$ satisfies
$$
\mathcal{L}\underline{v}-\underline{v}^p= \tau t^{\alpha_+}\left\{\Lambda\left(-\tfrac{2}{p-1}\right)-\Lambda(\alpha_+)-\tau^{p-1}t^{\alpha_+(p-1)}\right\}.
$$
Consequently for small $\tau>0$ the expression above is positive and $\underline{v}$ is a therefore a lower solution of \eqref{eq:v}.
\medskip

In a second step we shall construct an upper solution of \eqref{eq:v}. For this purpose consider the function
$$
\overline{v}=ct^\alpha(1-\kappa t^\epsilon), \tx{where} \alpha=\alpha_+,\: \kappa<1 \tx{and} \epsilon >0.
$$
A straightforward computation yields
\begin{equation*}
\mathcal{L} \overline{v}=c\Lambda_0t^\alpha-c\kappa\epsilon(2\alpha +\epsilon -1)t^{\alpha +\epsilon -2} -c\kappa \Lambda_\epsilon t^{\alpha +\epsilon},
\end{equation*}
where
$$\Lambda_\epsilon:=\Lambda(-\tfrac{2}{p-1})-\Lambda(\alpha +\epsilon).$$
Thus
$$
\mathcal{L}\overline{v}-\overline{v}^p=ct^{\alpha +\epsilon-2}[\Lambda_0t^{2-\epsilon}-\kappa \epsilon(2\alpha+\epsilon-1)-\kappa\Lambda_\epsilon t^2-c^{p-1}t^{\alpha(p-1)+2-\epsilon}(1-\kappa t^\epsilon)^p].
$$
If we choose $\epsilon <2$, $\kappa<1$ and keep in mind that  $2\alpha_+>1$ and that $\Lambda_0,\Lambda_\epsilon >0$, the expression above is negative for large
$c$. Hence $\overline{v}$ is an upper solution. We can always take $\tau$ sufficiently small and $c$ sufficiently large such that $\underline{v}<\overline{v}$. By the method of upper and lower solutions there exists a solution  $\underline{v}\leq v\leq \overline{v}$.
\begin{lemma}\label{existence1}
(i) If $1<p<p_c$ then \eqref{eq:v} has a  solution $v$ such that for some $0<c_1<c_2$,
$$c_1 t^{\alpha_+}<v(t)<c_2t^{\alpha_+}\tx{for all $\,t\in [0,1]$.}$$
(ii) If $p\ge p_c$ then there is no such solution. In particular there is no solution of \eqref{eq:v} such that $\int_0^1\sigma \dot{v}^2dt$ and $\int_0^1\frac{v^2}{t^2}\:dt$ exist.
\end{lemma}
\proof
The existence has been established above by means of the method of upper and lower solutions. In order to prove (ii) we assume that there exists a positive solution $v$.
Testing \eqref{eq:v} with $t^{\alpha_+}$ we obtain
\begin{equation}\label{eq-nonex}
\int_0^1\frac{\sigma}{1-t^2}\left(vt^{\alpha_+}[\Lambda(-\tfrac{2}{p-1})-\Lambda(\alpha_+)] -v^pt^{\alpha_+}\right)\:dt=0.
\end{equation}
By our assumption on $p$ the bracket $[\Lambda(-\tfrac{2}{p-1})-\Lambda(\alpha_+)] $ is non positive. This contradicts the identity \eqref{eq-nonex} unless $v\equiv 0$. \hfill $\square$
\medskip

\subsection{Construction of a solution $v$ that behaves like $t^{\alpha_-}$ near $t=0$}
 The construction of an upper  solution for small solutions in Lemma \ref{existence1} relies heavily on the fact that $\alpha_+>1/2$. Since $\alpha_-<1/2$, we need
a different argument. Throughout this section we shall make the additional assumption
\begin{align}\label{as:3}
-\tfrac{2}{p-1}<\alpha_-,
\end{align}
which is equivalent with $1<p<p_{KO}$ defined in \eqref{de:p}.
We shall distinguish between two cases.

(i) $\Lambda\left(-\tfrac{2}{p-1}\right)-\Lambda(\alpha_-)\leq 0.$

From \eqref{identity} it follows that the function $\overline{v}:=ct^{\alpha_-}$ satisfies
\begin{align}\label{eq:q}
\mathcal{L}\overline{v}-\overline{v}^p= c t^{\alpha_-}\left\{\Lambda\left(-\tfrac{2}{p-1}\right)-\Lambda(\alpha_-)-c^{p-1}t^{\alpha_-(p-1)}\right\}.
\end{align}
By our assumption the right-hand side is non positive for $t\in [0,1]$. The function $\overline{v}=ct^{\alpha_-}$ is therefore an upper solution.

For a lower solution we make the ansatz $\underline{v}=\tau t^{\alpha_-}(1- \kappa t^\epsilon)_+$, $\kappa\geq 1$. Then
$$
\mathcal{L}\underline{v}-\underline{v}^p=\tau t^{\alpha_- +\epsilon-2}[\Lambda_0^-t^{2-\epsilon}- \kappa \epsilon(2\alpha_-+\epsilon-1)-\kappa \Lambda_\epsilon^- t^2-\tau^{p-1}t^{\alpha_-(p-1)+2-\epsilon}(1- \kappa t^\epsilon)^p] \tx{in} \Big(0,\frac{1}{\kappa^{\frac{1}{\epsilon}}}\Big).
$$
Here
\begin{align*}\Lambda_\epsilon^-=&\Lambda(-\frac{2}{p-1})-\Lambda(\alpha_- +\epsilon)= \left( -\frac{2}{p-1} +\alpha_-+\epsilon +n-2\right) \left(-\frac{2}{p-1}-\alpha_--\epsilon\right)\\
=& \Lambda_0^- -\epsilon(2\alpha_-+n-2)-\epsilon^2.
\end{align*}
By our assumptions,
$\Lambda_0^-\leq 0$ and \eqref{as:3}, we have $-\frac{2}{p-1} +\alpha_- +n-2\geq 0$.  Thus $\Lambda^-_\epsilon <\Lambda_0^-$ for any positive $\epsilon$. Let $\epsilon >0$ be such that $2\alpha_-+\epsilon -1<0$ and $\alpha_-(p-1) +2-\epsilon >0$
(the latter strict inequality is satisfied in view of \eqref{as:3}). It is now possible to choose $\kappa>1$ such that
\begin{align*}
\Lambda_0^-t^{2-\epsilon}- \kappa \epsilon(2\alpha_-+\epsilon-1)-\kappa \Lambda_\epsilon^- t^2>0 \tx{in} \Big[0,\frac{1}{\kappa^{\frac{1}{\epsilon}}}\Big].
\end{align*}
Then $\tau$ can be taken so small that $\mathcal{L}\underline{v}-\underline{v}^p\geq 0$ in $(0,\kappa^{-1/\epsilon}]$. Moreover we have $\mathcal{L}\underline{v}-\underline{v}^p= 0$
in $(\kappa^{-1/\epsilon},1)$. Hence $\underline{v}$ is a lower solution in the weak sense.  We can now take $c$ in the definition of the upper solution so large that
$\underline{v}<\overline{v}$. By standard arguments there exists a solution $\underline{v}\leq v\leq ct^{\alpha_-}$.
\medskip

(ii) $\Lambda\left(-\tfrac{2}{p-1}\right)-\Lambda(\alpha_-)> 0.$

Now there holds $\Lambda_0^->0$. In this case we have to modify the upper solution.
We make the ansatz $\overline{v}= c t^{\alpha_-}(1+ \kappa t^\epsilon)$, $\kappa>1$. Then
$$
\mathcal{L}\overline{v}-\overline{v}^p=c t^{\alpha_- +\epsilon-2}[\Lambda^-_0t^{2-\epsilon}+ \kappa \epsilon(2\alpha_-+\epsilon-1)+\kappa \Lambda^-_\epsilon t^2-c^{p-1}t^{\alpha_-(p-1)+2-\epsilon}(1+\kappa t^\epsilon)^p] \tx{in} (0,1).
$$
We choose again $\epsilon$ such that $2\alpha_-+\epsilon -1<0$  and $\alpha_-(p-1) +2 -\epsilon >0$. For large $c$ and $\kappa$ the expression in the brackets is negative and $\overline{v}$ is an upper solution. For the lower solution we set $\underline{v}=\tau t^{\alpha_-}(1-\kappa t^\epsilon)_+$, $\kappa>1$ and choose $\epsilon$ as for the upper solution. As in (i) we deduce that for small $\tau$, $\underline{v}$ is a weak lower solution which is bounded from above by $\overline{v}$. By the same arguments as before there exists a solution $\underline{v}\leq v\leq\overline{v}$.
\subsubsection{Review of the upper and lower solutions of \eqref{eq:v}}
We summarize the different upper and lower solutions constructed in the previous sections, in dependence on the parameters $\mu$ and $p$.

In case (i) of Section 3.2, the assumptions $\Lambda_0^-\leq 0$ and \eqref{as:3} together imply firstly that $2\alpha_- +n-2>0$ or equivalently
\begin{equation*}
\mu^*:=-\frac{n(n-2)}{4}< \mu.
\end{equation*}
and secondly
\begin{equation*}
p_{c}^-:=
1+\frac{2}{n-2+\alpha_-}\leq p.
\end{equation*}
Therefore, for $\mu\le\mu^*$ it holds $p_{c}^-\ge p_{KO}$ and the case (i) is empty.

In case (ii) of Section 3.2, the assumptions $\Lambda_0^-> 0$ and \eqref{as:3}  imply
$$
-\frac{2}{p-1} +\alpha_-+n-2 <0.
$$
Hence either
$$
\mu^*<\mu \tx{and} 1<p<p_c^-,
$$
or
$$
\mu^*\geq \mu \tx{and} 1<p<p_{KO}.
$$
If $\mu=\mu^*$, then $p^-_c =\infty$. The upper and lower solutions constructed in this section are illustrated in Table \ref{table1}.

\begin{table}[h!]
\begin{center}
\begin{tabular}{c c c c }
\toprule
$\mu$
& $p$ & subsolution & supersolution \\
\midrule
$-\infty<\mu<1/4$ &$1<p< p_c$& $\tau t^{\alpha_+}$ & $ct^{\alpha_+}(1-\kappa t^\epsilon)$, $\;\kappa<1$ \\
\addlinespace
$\mu^*<\mu<1/4$ &$p_c^-
\le p<p_{KO}$ &$\tau t^{\alpha_-}(1-\kappa t^\epsilon)_+$, $\;\kappa>1$ & $ct^{\alpha_-}$\\
\addlinespace
$\mu^*<\mu<1/4$ & $1<p<p_c^-$ &$\tau t^{\alpha_-}(1-\kappa t^\epsilon)_+$,\: $\kappa>1$& $ct^{\alpha_-}(1+ \kappa t^\epsilon)$, $\;\kappa\gg 1$\\
\addlinespace
$-\infty<\mu\leq\mu^*$ &$1<p<p_{KO}$&$\tau t^{\alpha_-}(1-\kappa t^\epsilon)_+,\: \kappa >1$& $ct^{\alpha_-}(1+ \kappa t^\epsilon)$, $\;\kappa\gg 1$\\
\bottomrule
\end{tabular}
\end{center}
\caption{Sub and supersolutions for \eqref{P} ($\tau$ is small, $c$ is large)}\label{table1}
\end{table}

In conclusion we have the following.
\begin{lemma}\label{existence2}
Assume $1<p<p_{KO}$. Then there exist constants $0<c_1<c_2$ such that  \eqref{eq:v} has a solution $v$ satisfying
$$c_1t^{\alpha_-}<v(t)<c_2t^{\alpha_-} \tx{for all} t\in[0,1].$$
\end{lemma}

\begin{remark}
In the case $\mu=0$ we have $\alpha_-=0$. Then for $1<p<\frac{n}{n-2}=p_c^-$,
$$v(t)=\Lambda\left(-\tfrac{2}{p-1}\right)^\frac{1}{p-1}$$
is a solution as in Lemma \ref{existence2}, see \eqref{eq:q}.
\end{remark}

\subsection {Uniqueness} This section is devoted to the proof of the following

\begin{proposition}\label{le:uniqueness}
Assume $\Lambda_0=\Lambda(-\frac{2}{p-1})-\Lambda(\alpha_+)>0$.
Then \eqref{eq:v} has a unique positive solution $v$  that satisfies $v\leq c_1t^{\alpha_+}$
for $t\in[0,1]$ and  for some $c_1>0$.
\end{proposition}
We start with an auxiliary result.
Set $v=t^{\alpha_+}w.$ A straightforward computation implies that $w$ satisfies
\begin{align*}
(1-t^2)\ddot{w}+[\frac{2\alpha_+}{t}(1-t^2) -(n-1)t]\dot{w}+[\Lambda\left(-\tfrac{2}{p-1}\right)-\Lambda(\alpha_+)]w=t^{\alpha_+(p-1)}w^p,
\end{align*}
or equivalently
\begin{align}\label{eq:w}
(\tilde \sigma \dot{w})_t=\frac{\tilde\sigma}{1-t^2} \left[t^{\alpha_+(p-1)}w^p-\Lambda_0w\right], \tx{with} \tilde \sigma(t)= t^{2\alpha_+}(1-t^2)^{\frac{n-1}{2}}.
\end{align}
\begin{lemma}\label{pr:as}
If  $\Lambda_0>0$ and $w(0)<\infty$, then $w(0)>0$.
\end{lemma}
\proof
From
$$
\tilde\sigma(t)\dot{w}(t)-\tilde\sigma(\epsilon)\dot{w}(\epsilon)= \int_\epsilon^t\frac{\tilde\sigma}{1-s^2} \left[s^{\alpha_+(p-1)}w^p-\Lambda_0w\right]\:ds
$$
we deduce that $\lim_{\epsilon\to 0}\tilde\sigma(\epsilon)\dot{w}(\epsilon)=b$. Since $w$ is bounded in $t=0$, it follows that $b=0$.
Thus
$$
w(t)-w(0)= \int_0^t\tilde\sigma^{-1}\:ds \int_0^s\frac{\tilde\sigma}{1-\xi^2} \left[\xi^{\alpha_+(p-1)}w^p-\Lambda_0w\right]\:d\xi.
$$
If $w(0)=0$ the right-hand side is negative for small $t$ and therefore also $w(t)$. This is impossible and consequently $w(0)>0$. \hfill $\square$

\begin{proof}[Proof of Proposition \ref{le:uniqueness}]
The existence has been established in Lemma \ref{existence1}. In order to prove uniqueness we follow the arguments of Lemma 3.2 in \cite{BaMoRe08} (see also Theorem 6.18  in \cite{Du06}). Suppose that there are two positive solutions $V$ and $v$ of \eqref{eq:v} that satisfy $V\leq c_1t^{\alpha_+}$
and $v\leq c_2t^{\alpha_+}$ for $t\in[0,1]$.
We first treat the case where
$V(t)>v(t)$ in $(a,b)\subset (0,1)$, $V(t)=v(t)$  at the endpoints $t=a>0$ and $t=b<1$.  Set
$V(t)=t^{\alpha_+}W(t)$ and $v(t)=t^{\alpha_+}w(t)$. Both $W$ and $w$ satisfy equation \eqref{eq:w}.  We write $W-w=:\phi w$ and observe that $\phi(a)=\phi(b)=0$ and that
$$
\frac{d}{dt}\left(\tilde\sigma( w\phi)_t\right)=\frac{\tilde\sigma}{1-t^2} \left[t^{\alpha_+(p-1)}(W^p-w^p)-\Lambda_0w\phi\right].
$$
If we multiply this equation with $w\phi$ and integrate, we obtain
\begin{align}
\label{eq:w2}
\int_a^b \frac{\tilde\sigma}{1-t^2} \left[t^{\alpha_+(p-1)}(W^p-w^p)w\phi-\Lambda_0w^2\phi^2\right]\:dt&=-\int_a^b\tilde \sigma [(w\phi)_t]^2\:dt\\
\nonumber &= -\int_a^b\tilde \sigma(w^2\dot\phi^2 +2w\phi\dot w\dot \phi +\dot w^2\phi^2)\:dt.
\end{align}
Multiplying \eqref{eq:w} by $w\phi^2$ and integrating, we get
\begin{align}\label{eq:w3}
\int_a^b \frac{\tilde\sigma}{1-t^2} \left[t^{\alpha_+(p-1)}w^{p+1}\phi^2-\Lambda_0w^2\phi^2\right]\:dt&=-\int_a^b\tilde \sigma(2w\phi\dot w \dot \phi + \dot w^2\phi^2)\:dt
\end{align}
From \eqref{eq:w2} and \eqref{eq:w3} it follows that
\begin{align}\label{eq:w4}
-\int_a^b \tilde \sigma \dot \phi^2w^2\:dt = \int_a^b \frac{\tilde \sigma}{1-t^2} t^{\alpha_+(p-1)} [(W^p-w^p)w\phi -w^{p+1}\phi^2]\:dt.
\end{align}
Since
$$
(W^p-w^p)w\phi\geq pw^{p-1}(W-w)w\phi= pw^{p+1}\phi^2,
$$
\eqref{eq:w4} implies that $W=w$, contradicting our assumption.

Suppose now that $W-w>0$ in $(0,1)$. By Lemma \ref{pr:as} there exists a positive function $\phi$ such that $W-w=w\phi$.
In the proof of Lemma \ref{pr:as} it was shown that $\lim_{t\to 0}\tilde \sigma(t)\dot w(t)=0$.  Consequently
$$
\tilde\sigma(t)\dot{w}(t)= \int_0^t\frac{\tilde\sigma}{1-s^2} \left[s^{\alpha_+(p-1)}w^p-\Lambda_0w\right]\:ds,
$$
which, by the rule of Bernoulli-L'Hospital, implies that $\dot w(0)=0$. Hence the integrals at the right-hand side of \eqref{eq:w2} exist at $t=0$.  By the same argument
they also exist at $t=1$.
Since $\tilde \sigma(\pm 1)=0$,
\eqref{eq:w2} and \eqref{eq:w4} hold if we replace $a$ by $0$ and $b$ by $1$. Exactly in the same way we treat the cases where $a=0$, $b<1$ and $a>0$ and
$b=1$.
\end{proof}

\medskip

The  investigations of this section lead to the following results for the solutions of nonlinear problem \eqref{P}.
\begin{theorem}\label{thm:sep}
Assume that $-\infty<\mu<1/4$.

(i) If $1<p<p_c$ then \eqref{P} has a unique separable solution of the form
\begin{equation}\label{e-separ}
u(r,\theta_1)=r^{-\frac{2}{p-1}}v(\cos(\theta_1))
\end{equation}
where $0<v<\infty$  satisfies
$$\lim_{\theta_1\to \pi/2}\frac{v(\cos(\theta_1))}{\cos(\theta_1)^{\alpha_+}}=v(0)>0.$$

(ii) If $p\ge p_c$ then \eqref{P} has no separable solution of the form \eqref {e-separ} that satisfies
$$
c_1 \cos(\theta_1)^{\alpha_+}<v(\cos(\theta_1)) <c_2\cos(\theta_1)^{\alpha_+} \tx{in} [0,\pi/2].
$$
\end{theorem}

We have therefore found a solution that has a strong singularity at the origin and behaves like a small $\L_\mu$-harmonic at the rest of the boundary.
In the next theorem we describe a solution with a strong singularity at the origin that behaves like a large $\L_\mu$-harmonic at the rest of the boundary.

\begin{theorem}
Assume that $-\infty<\mu<1/4$. If $1<p<p_{KO}$, then there exist constants $0<c_1<c_2$ such that \eqref{P}  has a separable solution of the form
$$u(r,\theta_1)=r^{-\frac{2}{p-1}}v(\cos(\theta_1))$$
satisfying
$$
c_1 \cos(\theta_1)^{\alpha_-}<v(\cos(\theta_1)) <c_2\cos(\theta_1)^{\alpha_-} \tx{in} [0,\pi/2].
$$
\end{theorem}
\begin{remark}
$(i)$
We expect that in this case the separable solution is not unique.

$(ii)$ From the Keller--Osserman estimate, which is presented in the next section, it follows that this result is sharp, in the sense that no such solutions exist if $p>p_{KO}$.
\end{remark}
\section{Phragm\'en--Lindel\"of type estimate, Keller--Osserman a priori bound and nonexistence proofs}\label{s-PL}

We establish a version of the Phrag\-men--Lin\-del\"of type comparison
principle, which shows in particular that a class of $\L_\mu$-subharmonics
either have a prescribed order of singularity at the origin or have a ``regular'' decay at the origin,
similarly to the dichotomy exhibited by the $\L_mu$-harmonics in \eqref{harmonics1}.
See \cite[pp. 93-106]{Protter} for a classical reference to the Phragm\'en--Lindel\"of principle.

\begin{lemma}\label{l-PL}
{\sc (Phragm\'en--Lindel\"of type estimate)}
Let $\mu< 1/4$. Let $h\in C^2(\R^n_+\cap B_R(0))$ be an $\L_\mu$-subharmonic in $\R^n_+\cap B_R(0)$,
for some $R>0$. Assume that $x\in\R^n_+\cap B_R(0)$ and
$$\exists\rho\in(0,R):\quad
\limsup_{x_1\to 0}\frac{h(x)}{x_1^{\alpha_+}}<+\infty\;\tx{as $\rho<|x|<R$,}\leqno{(a)}$$
$$\lim_{x_1\to 0}\frac{h(x)}{x_1^{\alpha_-}+x_1^{\alpha_+}|x|^{-(n-2+2\alpha_+)}}=0.\leqno{(b)}$$
Then for $x\in\R^n_+\cap B_{R}(0)$ it holds
\begin{equation*}
\limsup_{x_1\to 0}\frac{h(x)}{x_1^{\alpha_+}}<+\infty.
\end{equation*}
\end{lemma}

\begin{proof}
Without loss of generality we may assume that $h$ is continuous on $\R^n_+\cap \bar{B}_R(0)$.
Choose $C_\rho>0$ such that
$$h< C_\rho x_1^{\alpha_+}\;\tx{as $\rho<|x|<R$.}$$
For $\tau>0$, define a comparison function
$$h_\tau:=h-C_\rho x_1^{\alpha_+}-\tau\big(x_1^{\alpha_-}+x_1^{\alpha_+}|x|^{-(n-2+2\alpha_+)}\big).$$
Clearly, $h_\tau$ is $\L_\mu$-subharmonic in $\R^n_+\cap B_R(0)$ and $h_\tau\in C^2_{loc}(\R^n_+\cap B_R(0))$.

For every $\tau>0$, conditions $(a)$ and $(b)$ imply that  $h_\tau\le 0$ on a neighbourhood of $\partial(\R^n_+\cap B_R(0))$.
Hence we can apply the classical comparison principle for $\L_\mu$, which applies for any $\mu<1/4$ in proper subdomains of $\R^n_+\cap B_R(0)$ (see for instance \cite[Lemma 2.4]{BaMoRe08}) and deduce that $h_\tau\le 0$ everywhere in $\R^n_+\cap B_R(0)$.
By considering arbitrary small $\tau>0$, we conclude that $h\le C_\rho x_1^{\alpha_+}$ in $\R^n_+\cap B_R(0)$.
\end{proof}

Next we prove a localised version of a Keller--Osserman type bound for positive solutions of the nonlinear problem \eqref{P}.

\begin{lemma}[\sc Keller--Osserman type bound]\label{p-KO}
Let $\mu<1/4$ and $p>1$. Let $u$ be an arbitrary positive solution of \eqref{P}
in $\R^n_+\cap B_R(0)$, for some $R>0$. Then
\begin{equation}\label{e-KO}
u(x)\le Cx_1^{-\frac{2}{p-1}}\quad\text{in $\R^n_+\cap B_{R/2}(0)$,}
\end{equation}
where $C>0$ is a universal constant that depends on $R$ but does not depend on $u$.
\end{lemma}
\begin{proof}
The argument is  based, as in  the proof of Lemma 35 in \cite{BaMoRe08}, on the construction of a supersolution $U$ satisfying
$$
\Delta U+\frac{\mu}{x_1^2}U\leq U^p \tx{in}\{x_1>\epsilon\}\cap B_R(0),\quad U=\infty \tx{on} \partial \left(\{x_1>\epsilon\}\cap B_R(0)\right) .
$$
We set $U=c(x_1-\epsilon)^{-\frac{2}{p-1}}(R-r)^{-\frac{2}{p-1}}$.  Then
\begin{align*}
\Delta U&= c(x_1-\epsilon)^{-\frac{2}{p-1}}\Delta (R-r)^{-\frac{2}{p-1}} -2c\left(\frac{2}{p-1}\right)^2(x_1-\epsilon)^{-\frac{2}{p-1}-1}(R-r)^{-\frac{2}{p-1}-1}\frac{x_1}{r} \\
&+ c(R-r)^{-\frac{2}{p-1}}\Delta (x_1-\epsilon)^{-\frac{2}{p-1}}.
\end{align*}
Furthermore
\begin{align*}
\Delta(R-r)^{-\frac{2}{p-1}}&= 2\frac{(p+1)}{(p-1)^2}(R-r)^{-\frac{2p}{p-1}}+\frac{2(n-1)}{r(p-1)}(R-r)^{-\frac{2}{p-1}-1}\\
&=: A_p(R-r)^{-\frac{2p}{p-1}}+B_p\frac{(R-r)^{-\frac{2}{p-1}-1}}{r}\\
\Delta(x_1-\epsilon)^{-\frac{2}{p-1}}&= 2\frac{(p+1)}{(p-1)^2}(x_1-\epsilon)^{-\frac{2p}{p-1}}=A_p(x_1-\epsilon)^{-\frac{2p}{p-1}},
\end{align*}
which leads to
\begin{multline*}
\L_\mu U-U^p= c(x_1-\epsilon)^{-\frac{2p}{p-1}}(R-r)^{-\frac{2p}{p-1}}\times\\
\left[(x_1-\epsilon)^2(A_p +B_p\frac{R-r}{r})-C_p(x_1-\epsilon)(R-r) \frac{x_1}{r}\right.
\\\left.+A_p(R-r)^2+\frac{\mu}{x_1^2}(x_1-\epsilon)^2(R-r)^2-c^{p-1}
\right],
\end{multline*}
where $C_p=2\left(\frac{2}{p-1}\right)^2$.
It is easy to see that for large $c$ the expression in the brackets is negative for all $\epsilon <\epsilon_0$. Consequently $U$
is a supersolution in $\R^n_+\cap B_R(0)$.

Observe that $U,u\in C^2_{loc}(\R^n_+\cap B_R(0))$.
By the classical comparison principle for $\L_\mu$, which applies for any $\mu<1/4$ in proper subdomains of $\R^n_+\cap B_R(0)$
(see for instance \cite[Lemma 2.4]{BaMoRe08})
it follows that $u\leq U$ in $\R^n_+\cap B_R(0)$
and $u< cx_1^{-\frac{2}{p-1}}(R/2)^{-\frac{2}{p-1}}$ in $\R^n_+\cap B_{R/2}(0)$.
\end{proof}

Next we establish the following nonexistence result which extends Theorem \ref{thm:sep} (ii) to non separable solutions.

\begin{proposition}\label{lemma-nonexistence}
Let $\mu<1/4$ and $p>p_c:=1+\frac{2}{n-2+\alpha_+}$. Then for any $R>0$, equation \eqref{P} admits no positive solutions $u$ in $\R^n_+\cap B_R(0)$ which satisfy
$$u(x)\le c\,x_1^{\alpha_+}|x|^{-\alpha_+ - \frac{2}{p-1}}\;\tx{in $\R^n_+\cap B_R(0)$,}\leqno{(a)}$$
$$\liminf_{x\to 0}\frac{u(x)}{|x|^{-\frac{2}{p-1}}}>0\;\tx{as $x\to 0$ nontangentially in $\R^n_+$.}\leqno{(b)}$$
\end{proposition}

\begin{proof}
Since $p>1+\frac{2}{n-2+\alpha_+}$, using condition $(a)$ and $\alpha_-<\alpha_+$, we conclude that for $x\in\R^n_+\cap B_{R}(0)$,
\begin{equation*}
\lim_{x_1\to 0}\frac{u(x)}{x_1^{\alpha_-}+x_1^{\alpha_+}|x|^{-(n-2+2\alpha_+)}}=0.
\end{equation*}
Then, by Lemma \ref{l-PL} with $\rho=R/2$, we conclude that
\begin{equation*}
\lim_{x_1\to 0}\frac{u(x)}{x_1^{\alpha_+}}<\infty.
\end{equation*}
But this contradicts  $(b)$.
\end{proof}

Theorem \ref{thm-nonexist} now follows as a special case of Proposition \ref{lemma-nonexistence},
while Theorem \ref{t-non-KO} follows directly from Lemma \ref{p-KO}
and \eqref{e-minus}.



\section{Solutions with a general singular set: proof of Theorem \ref{thm-general}}\label{s-general}

In this section we construct solutions of the nonlinear problem \eqref{P} that have a strong singularity
on an arbitrary closed subset $F$ of the boundary, and that behave either as $x_1^{\alpha_+}$ or as $x_1^{\alpha_-}$ as
$x\to (0,\xi)\in\partial \mathbb{R}^n_+\setminus F$ non-tangentially. Here $x=(x_1,\xi)$ and $\xi=(x_2,\ldots, x_n)$.

We start with some notation. Denote by $K_\mu$ the Martin kernel of $\L_\mu$ in $\R^n_+$. It is known that
\begin{equation*}
K_\mu(x,y)\sim x_1^{\alpha_+}|x-y|^{2\alpha_- -n}\quad \forall x\in \R^n_+,\, y\in \partial\R^n_+.
\end{equation*}
Let $K_{\mu,R}$ be the Martin kernel of $\L_\mu$ in $D_R:=B_R(0)\cap \R^n_+$. Then
$$K_{\mu,R}\to K_\mu \quad \text{as }R\to\infty,$$
uniformly on compact subsets of $\R^n_+$.
By the representation theorem of Ancona \cite{Ancona}, every positive $\L_\mu$-harmonic function $u$ in $D_R$ can be represented in the form
 \begin{equation}\label{Rep,DR}
 u(x)=\int_{D_R}K_{\mu,R}(x,y)d\nu(y),
 \end{equation}
where $\nu\in \M_+(\partial D_R)$ ($=$ space of positive finite Borel measures in $\partial D_R$). Conversely, for every $\nu$ as above the function $u$ defined in \eqref{Rep,DR} is $\L_\mu$-harmonic in $D_R$.

Denote,
$$T_{r,r'}:=\{x=(x_1,x')\in \R^n_+: 0< x_1< r,\; |x'|<r'\}.$$
If $(0,x')\in \partial\R^n_+$, put $T_{r,r'}(0,x')= (0,x')+T_{r,r'}$.
\medskip

\noindent\textit{Proof of Theorem \ref{thm-general}.}
Denote by $v_0$ the solution of \eqref{P} in $\R^n_+$, constructed in Theorem \ref{thm-main}. It satisfies
$$\lim_{x_1\to 0}\frac{v_0(x_1,0)}{x_1^{2/(p-1)}}= c_0.$$
Denote  $v_\xi=v_0(x+(0,\xi))$. Let $\{\xi_k\}$ be a dense sequence of points in $F$.
For every $m$, $\max(v_{\xi_1},\ldots, v_{\xi_m})$ is a subsolution and $v_{\xi_1}+\ldots, +v_{\xi_m}$ is a supersolution. Therefore there exists a solution $V_m$ of \eqref{P} in $\R^n_+$, such that
\[\max(v_{\xi_1},\ldots, v_{\xi_m})\leq V_m\leq v_{\xi_1}+\ldots, +v_{\xi_m}.\]
In fact there exists a minimal such solution and it is this solution that we denote by $V_m$.
Clearly the sequence $\{V_m\}$ increases and $$\liminf_{x_1\to 0}\frac{V_m(x_1,\xi_k)}{x_1^{2/(p-1)}}\geq c_0, \quad k=1,\ldots,m,$$
uniformly in the following sense: for every $\epsilon\in(0,1)$ there exists $a(\epsilon)>0$, independent of $m$, such that
$$ \frac{V_m(x_1,\xi_k)}{x_1^{2/(p-1)}}\geq (1-\epsilon)c_0 \quad   k=1,\ldots m,\; 0<x_1<a(\epsilon).$$
Obviously, the solution $V=\lim V_m$ satisfies the same inequality,
$$\frac{V(x_1,\xi_k)}{x_1^{2/(p-1)}}\geq (1-\epsilon)c_0 \quad k=1,2,\ldots,\; 0<x_1<a(\epsilon).$$
Because of continuity of $\xi \mapsto V(x_1,\xi)$ for every fixed $x_1>0$,
$$\frac{V(x_1,\xi)}{x_1^{2/(p-1)}}\geq (1-\epsilon)c_0 \quad \text{for every}\quad (0,\xi)\in F,\;    0<x_1<a(\epsilon).$$
This inequality and the Keller--Osserman estimate imply  \eqref{gen1}.

For any point $y\in \R^n$ put
$$d_F(y)=\mathrm{dist}(y,F).$$
Clearly,
\begin{equation*}
 v_\xi(x)=v_0((0,\xi)+x)\leq V(x) \qquad \forall \xi\in F,\; x\in \R^n_+.
\end{equation*}
This implies the left hand inequality in \eqref{gen2}. We turn to the proof of the right hand inequality.

For every $a>0$ and every solution $u$ of (${P}_\mu$) in $\R_+^N$, put
$$u_a(x)=a^{\frac{2}{p-1}}u(ax).$$
Then $u_a$ too is a solution of (${P}_\mu$).
If $F$ is the set of strongly singular points of $u$  then $F^a:= \frac{1}{a}F$ is the set of strongly singular points of $u_a$. If $P\in \bar\R^n_+\setminus F$, we denote by $u^P$ the function $u_a$ with $a=d_F(P)/4$. Let $V$ and $V_k$ be defined as before.

 Let $P\in \partial\R^n_+\setminus F$ and $a=d_F(P)/4$. We shall prove that there exists a constant $C$ depending only on $\mu,c,p$ such that
\begin{equation}\label{egen1}
 V_k^P(z)\leq C z_1^{\alpha_+} \quad\forall z\in T_{1,1}(P_a), \quad P_a=\frac{1}{a}P.
\end{equation}
As $C$ is independent of $k$, \eqref{egen1} implies that
\begin{equation*}
 V^P(z)\leq C z_1^{\alpha_+} \quad\forall z\in T_{1,1}(P_a), \quad P_a=\frac{1}{a}P.
\end{equation*}
Since
$$V^P(z)=a^{\frac{2}{p-1}}V(az),$$
the last inequality is equivalent to
\begin{equation*}
V(x)=a^{-\frac{2}{p-1}}V^P(x/a)\leq C a^{-\frac{2}{p-1}}(x_1/a)^{\alpha_+}
\leq C'd_F(x)^{-1-\frac{2}{p-1}}x_1^{\alpha_+}
\quad\forall x\in T_{a,a}(P)
\end{equation*}
where $C'=8^{1+\frac{2}{p-1}}C$. Thus \eqref{egen1} implies the right hand side inequality in \eqref{gen2}.

Now to prove \eqref{egen1}, we fix $k$ and $P$ and put $w=V_k^P$. Since $d_{F^a}(P^a)=4$, $w$ vanishes continuously at $z_1=0$ for $|z'-P_a|<4$. By the Keller--Osserman estimate,
$w(z)\leq c(\mu,n,p)z_1^{-\frac{2}{p-1}}$. In particular
$w<c'=c2^{\frac{2}{p-1}}$ in the strip $1/2<z_1<2$. Let $V_\mu^w=-\frac{\mu}{z_1^2} + w^{p-1}$, $\L_\mu^w= \Delta +V_\mu^w$ and
$$-\L_\mu^w w=0\;\text{in}\; \R^n_+.$$
Obviously $|V_\mu^w|< c_0 z_1^{-2}$ -- where $c_0$ is independent of $w$ -- and there exists a positive $\L_\mu^w$-superharmonic function, e.g. a positive eigenfunction of $\L_\mu$. Therefore the results of Ancona \cite{Ancona} -- specifically the boundary Harnack principle (briefly BHP) -- may be applied to $\L_\mu^w$. Let $G_\mu^w$ denote the Green kernel of $\L_\mu^w$ in $\R^n_+$. Let $(0,z')$ be a point such that $|z'-P_a|<1$ and let $Z^*=(2,z')$, Z=(1,z').
Applying BHP to the pair $w$ and $G_\mu^w(\cdot,Z^*)$ in $T_{3/2,2}(0,z')$ we obtain
$$ \frac{w(Z)}{w(z)}\sim \frac{G_\mu^w(Z,Z^*)}{G_\mu^w(z,Z^*)} \quad \forall z\in T_{1,1}(0,z').$$
Hence
$$w(z)\leq C_1 \frac {w(Z)}{G_\mu^w(Z,Z^*)}G_\mu^w(z,Z^*)\leq C_2 G_\mu^w(z,Z^*)\leq C_2G_\mu(z,Z^*)\leq C z_1^{\alpha_+}.$$
Here we used the fact that $G_\mu^w\leq G_\mu$.
\qed

\noindent\textit{Notation.} If $E\subset \R^{n-1}=\partial\R^n_+$ and $\beta>0$, we denote
$$E_\beta=\{x=(\beta,x'): \mathrm{dist\,}(x',E)<\beta\}.$$
If $\tau$ is a positive finite measure on a Borel set $E$ as above, we denote
$$\mathbb{K}_\mu[\tau;E]=\int_{E} K_\mu(x,y)d\tau(y).$$

Following \cite{Marcus-Nguen} we say that a positive Borel function $u$ in $\R^n_+$ has normalized boundary trace $\tau$ on $E$ if $\tau(E)<\infty$ and
\begin{equation*}
\frac{1}{\alpha_-}\int_{E_\beta}|u-\mathbb{K}_\mu[\tau;E]|dS\to 0 \;\;\text{as}\;\beta\to0.
\end{equation*}

\begin{proposition}\label{P5.1}
Let $p<p_c$ and $\mu<1/4$. Let $F$ be an arbitrary compact set and let $\nu$ be a positive locally bounded measure on $\partial\R^n_+$. Then there exists a solution $u$ of \eqref{P} such that $u$ is strongly singular in the sense of \eqref{gen1} at every point of $F$ and has normalized boundary trace $\nu$ on  every compact subset of $F':=\partial\R^n_+\setminus F$.
\end{proposition}

\proof For every $R>0$ let $u_R$ be the solution of \eqref{P} in $D_R$ with normalized boundary
trace $\nu$ on $\partial_1 D_R:= (\partial D_R) \cap [x_1=0]$ and zero on $\partial_2 D_R = (\partial D_R)\cap\R^n_+$. (The existence and uniqueness of this solution is proved in \cite{Marcus-Nguen} for $0\leq \mu\leq 1/4$ and the proof is similar when $\mu<0$.) Then $u_R$ increases as $R$ increases and, because of the Keller Osserman estimate the limit $u_\nu:=\lim_{R\to\infty} u_R$ is a solution of \eqref{P} in $\R^n_+$ with normalized boundary trace $\nu$.

Let $U$ be a solution as in Theorem \ref{thm-general}. Then $\max(u_\nu,U)$ is $\L_\mu$-subharmonic and $u_\nu+u_F$ is $\L_\mu$-superharmonic. Therefore there exits an $\L_\mu$-harmonic function $u$ between these two. Clearly $u$ has the required boundary behavior.
\qed
\medskip

\begin{remark}
Let $1<p<p_{KO}$.
If $\nu=fdS$ where $f\in L^1(\R^{n-1})$ is a positive function then $w^f_R=\mathbb{K}_\mu[\nu; |x'|<R]$ increases with $R$ and $w_\nu=\lim_{R\to\infty}w^f_R$ is an $\L_\mu$-harmonic function in $\R^{n}_+$ with normalized boundary trace $\nu$.
It is not difficult to show that if, in addition, $f$ is continuous then $w_\nu (x)/x_1^{\alpha_-}\to cf$ when $x_1\to 0$ and $c$ is a constant independent of $f$.
It was shown in \cite{MM} that, for $\nu$ as above, the solution $u_\nu$ of $(P_\mu)$ exists for every $p\in (1,p_{KO})$.
Finally, $u_\nu/ w_\nu\to 1$, $\nu$-a.e. (see \cite {Marcus-Nguen}).
Therefore, $$\lim_{x_1\to0}u_\nu (x)/x_1^{\alpha_-}= cf$$ where $c$ is a constant independent of $f$.
\end{remark}

\begin{remark} When $p\geq p_c$, there may not exists any positive solution vanishing on $F'=\partial\R^n_+\setminus F$. Therefore the result stated in Theorem \ref{thm-general} does not extend to this case. However, if $1<p<p_{KO}$ then, following the construction of $V$  in the proof of Theorem \ref{thm-general} -- but with a function  $v_0$ satisfying \eqref{e-minus} -- one finds that
for every $y\in F$, $V(x)|x-y|^{2/(p-1)}$
converges to a positive constant $c_y$ as $x\to y$ non-tangentially.
Furthermore $V$ behaves like $x_1^{\alpha_-}$ at $F'$. More precisely, one obtains,
$$\frac{1}{c}x_1^{\alpha_-}\mathrm{dist}(x,F)^{-\alpha_- -\frac{2}{p-1}}\leq V(x)\leq c x_1^{\alpha_-}\mathrm{dist}(x,F)^{-\alpha_- -\frac{2}{p-1}}, \quad  \quad \forall x\in \R^n_+.$$
\end{remark}

\bigskip
\noindent\textbf{Acknowledgment.}
This work was initiated during a visit of CB and VM at the Technion.
Part of this work was conducted at the Isaac Newton Institute for Mathematical Sciences, Cambridge,
in the framework of "Free Boundary Problems and Related Topics" (2014) programme.
The support and hospitality of both institutions are gratefully acknowledged.

The authors are grateful to Yehuda Pinchover for many fruitful discussions and to an anonymous referee for their insightful comments.

\end{document}